\documentclass{amsart}
\usepackage{amscd,enumerate}
\usepackage{amssymb,amsmath,amsthm}
\usepackage{latexsym,amsfonts,amscd}

\usepackage[utf8]{inputenc}
\usepackage[OT4]{fontenc}

\newcommand*{\R}{\ensuremath{\mathbb{R}}}

\renewcommand*{\epsilon}{\varepsilon}

\newtheorem{theorem}{Theorem}
\newtheorem{definition}[theorem]{Definition}
\newtheorem{corollary}[theorem]{Corollary}
\newtheorem{lemma}[theorem]{Lemma}
\newtheorem{fact}[theorem]{Fact}

\newtheorem{example}[theorem]{Example}

\begin{document}
\title{Discontinuity points of a function with a closed and connected graph}
\author{Michał Stanisław Wójcik}
\maketitle

\begin{abstract}
The main result of this paper states that for a function $f:\R^2\to Y$ with a closed, connected and locally connected graph, where $Y$ is a locally compact, second-countable metrisable space, the graph over discontinuity points remains locally connected.
\end{abstract}

\subsection*{Motivation}
It is a classic result that for a function $f:X\to Y$ with a closed graph and $Y$ Hausdorff, a sufficient and necessary condition for being continuous is sub-continuity \cite[3.4]{Ful}. However, it is an interesting question how, for various spaces $X$ and $Y$, additional topological properties of a closed graph are related to the continuity of a function. It is known that, for a function $f:\R\to \R$ with a closed graph, a sufficient and necessary condition for being continuous is the connectedness of the graph \cite{Bur}. In 2001, Michał R. Wójcik and I stated the question whether this result can be extended to $f:\R^2\to\R$ \cite[9]{wojcik1} – this problem was then propagated by Cz. Ryll-Nardzewski. The answer to this question is negative. The first known discontinuous function $f:\R^2\to\R$ with a connected and closed graph was shown by J. Jel\'{\i}nek in \cite{JL}. It can be shown that the graph of Jel\'{\i}nek's function in not locally-connected \cite[A4]{mrwphd}. Therefore a new question was stated, whether connectedness together with the local connectedness of the graph is a sufficient and necessary condition of being continuous for a function $f:\R^2\to\R$ with a closed graph. This question, as far as I know, remains open. In this paper, I show some properties of the set of discontinuity of a function $f:\R^2\to\R$ with a closed, connected and locally connected graph, hoping that they might be useful in the main research.

\subsection*{Result}
The main result of this paper states that for a function $f:\R^2\to Y$ with a closed, connected and locally connected graph, where $Y$ is a locally compact, second-countable metrisable space, the graph over discontinuity points remains locally connected. This result is given as Corollary \ref{main} as a consequence of some deep topological properties of the real plane and the more generic Theorem \ref{cluster}.

\section{Notation and terminology}

\begin{definition}
Let $X,Y$ be topological spaces and $f:X\to Y$ be an arbitrary function.
\begin{enumerate} 
\item We will denote by $C(f)$ or $C_f$ the set of all points of continuity,
\item by $D(f)$ or $D_f$ -- set of all points of discontinuity. 
\item We will denote by $\pi$ a projection operator $\pi:X\times Y \to X$ and $\pi(x,y)=x$.
\item For $A\subset X$, by $f|A$ we will denote a restriction of $f$ to the subdomain $A$.
\end{enumerate} 
\end{definition}

In the context of function $f:X\to Y$, we will not use a separate symbol to denote the graph of $f$, for $f$ itself, in terms of Set Theory, is a graph. So when we use Set Theory operations and relations with respect to $f$, they should be understood as operations and relations with respect to the graph. Whenever this naming convention might be confusing, we will add the word “graph”, e.g. “$f$ has a closed graph”.

\begin{definition}
Let $Y$ be a topological space and $y_n\in Y$ be an arbitrary net. We will write $y_n \to \emptyset$ or $\lim y_n = \emptyset$ iff $y_n$ has no convergent subnet.
\end{definition} 

\section{Functions with a closed graph}

It will be helpful to cite two well-known theorems concerning functions with a closed graph:

\begin{theorem}\label{fcon}
If $X$ is a topological space, $Y$ is a compact space, $f:X\to Y$ and the graph of $f$ is closed, then $f$ is continuous.
\end{theorem}

(for proof: e.g. \cite[T2]{wojcik1})

\begin{theorem}
If $X$ is a Bair and Hausdorff space, $Y$ is a $\sigma$-locally compact space, $f:X\to Y$ and the graph of $f$ is closed, then C(f) is an open and dense subset of $X$.
\end{theorem}

(for proof: e.g. \cite[T2]{Dob})

\section{Graph over discontinuity points}

I will begin with several well-known facts:

\begin{fact}\label{cantorimg}
Every non-empty metrisable compact space is a continuous image of the Cantor set. 
\end{fact}
(for proof e.g. \cite[4.5.9]{Eng})

\begin{fact}\label{ext}
If $X$ is a connected and locally arcwise-connected metrisable space, $F$ is a closed subset of $X$, $C\subset[0,1]$ is the Cantor set and $f:C \to F$ is continuous and $f(C)=X$, then $f$ has a continuous extension $f^*:[0,1]\to X$.
\end{fact}
(for proof: e.g. \cite[50.I.5]{Ku2})

\begin{fact}\label{clo_lc}
If $X$ is a topological space and $A,B\subset X$ are both closed and locally connected, then $A\cup B$ is locally connected.
\end{fact}
(for proof: e.g. \cite[49.I.3]{Ku2})

\begin{fact}\label{lcimg}
If $X$ is a compact and locally connected space, $Y$ is a Hausdorff space, $f:X\to Y$ is continuous and $f(X)=Y$, 
then $Y$ is compact and locally connected. 
\end{fact}

(for proof: notice that since $X$ is compact and $Y$ is Hausdorff, $f$ is a quotient mapping
and local connectedness is invariant under quotient mappings \cite[T2]{Why})

\begin{lemma}\label{continuum_factory}
If $X$ is a connected metrisable space, and $F$ is a connected and closed subset, $X\setminus F = A \cup B$, where $Clo(A)\cap B = Clo(B) \cap A = \emptyset$, then $F\cup A$ is connected and closed.
\end{lemma}
\begin{proof}
\cite[46.II.4]{Ku2}
\end{proof}

\begin{lemma}\label{addcomponents}
If $X$ is a locally connected metrisable space, $F$ is a locally connected and closed subset and $S$ is a sum of some connected components of $X\setminus F$, then $S\cup F$ is locally connected.
\end{lemma}
\begin{proof}
\cite[49.II.11]{Ku2}
\end{proof}

\begin{theorem}\label{finitec}
If $X$ is a connected and locally connected, locally compact, second-countable metrisable space, 
$E$ is a continuum in $X$ and $U$ is an arbitrary open neighbourhood of $E$, 
then there exists a locally connected continuum $F$ such that $E\subset F \subset U$ and $X\setminus F$ 
has finitely many connected components. 
\end{theorem}
\begin{proof}
Since $X$ is locally compact, Hausdorff and locally connected, there exists an open neighbourhood $V_x$ of the point $x$ such that $Clo(V_x)$ is a continuum and $Clo(V_x)\subset U$ for each $x\in E$. 
Since $E$ is compact, there exist $x_1, x_2, \dots, x_n\in E$ such that $E \subset
\bigcup\limits_{i=1}^n V_{x_i}$. Let $V = \bigcup\limits_{i=1}^n V_{x_i}$. Notice that $Clo(V)$ is compact and $E \subset V \subset Clo(V)\subset U$. 
By Fact \ref{cantorimg}, there is a continuous function $f:C\to E$, where $C$ is the Cantor set and $f(C) = E$. Since $X$ is metrisable and locally compact, it is completely metrisable and therefore, by the Mazurkiewicz-Moore theorem, $X$ is locally arcwise-connected, and thus $V$ is locally arcwise-connected.
$V$ is connected by construction. Therefore by Fact \ref{ext} there is a function $f^*$ that is a continuous extension of $f$ such that $f^*:[0,1]\to V$.
Let $F_0 = f^*([0,1])$. By Fact \ref{lcimg}, $F_0$ is a locally connected continuum and $E\subset F_0\subset V$.
Let $\mathcal{S}_V$ be a family of all the connected components of $X\setminus F_0$ that are subsets of $V$. 
Let $\mathcal{S}_\infty$ be a family of all the other connected components of $X\setminus F_0$. Notice that, due to the connectedness of $S$, (1) $S\cap \partial V\not=\emptyset$ for any $S\in \mathcal{S}_\infty$. 
Since $X$ is locally connected, all the connected components of $X\setminus F_0$ are open in $X$. Since $\partial V\subset \bigcup\mathcal{S}_\infty$, by virtue of (1) and the compactness of $\partial V$, the family $\mathcal{S}_\infty$ is finite. 
Let $F = F_0 \cup \bigcup\mathcal{S}_V$. By Lemma \ref{continuum_factory}, $F$ is connected and closed. Since $F\subset Clo(V)$, $F$ is a continuum. By Lemma \ref{addcomponents}, $F$ is locally connected. Obviously, $E\subset F\subset V\subset U$. Since $X\setminus F =  \bigcup\mathcal{S}_\infty$ and $\mathcal{S}_\infty$ is finite, the proof is complete.
\end{proof}

\begin{lemma}\label{algebra}
If $X$ is a Hausdorff space, $Y$ is a topological space, $f:X\to Y$ is a function with a closed graph, $D=D(f)$, $E$ is a compact subset of the graph and $U$ is a relatively open subset of the graph such that $U\subset E \subset f$, 
then $U\cap f|D = U\cap f|\partial \pi(E)$. Moreover, if $(x,f(x))\in U \cap f|D$ and $X\setminus \pi(E)\ni x_n \to x$, then $f(x_n)\to \emptyset$.
\end{lemma}
\begin{proof}
Since $E$ is compact, by Theorem \ref{fcon}, $f|\pi(E)$ is continuous. If $x\in Int(E)$, then $f$ is continuous in $x$, so $x\not\in D$. Therefore $U\cap f|D \subset U\cap f|\partial \pi(E)$ is obvious. We will show inverse inclusion by contradiction. 
Assume that $(x, f(x))\in U \cap f|\partial \pi(E)$ and $x$ is a continuity point of $f$. Since $\pi(E)$ is compact and $X$ is Hausdorff, $\pi(E)$ is closed, so $x\in \pi(E)$. Notice that $(x,f(x))\in U$, so by the continuity of $f$ at point $x$, there is an open neighbourhood $V$ of $x$, such that $f|V \subset U$. But $U\subset E$, so $x\in V \subset \pi(E)$. 
This contradicts how $x$ was chosen.
Now we will show the “moreover part”. Take any $(x,f(x))\in U \cap f|D$ and $X\setminus \pi(E)\ni x_n \to x$. Since $U\subset E$, 
$(x_n, f(x_n))\not\in U$. So no subnet of $f(x_n)$ is convergent to $f(x)$. But the graph of $f$ is closed, so $f(x_n)\to\emptyset$.
\end{proof}

\begin{theorem}\label{cluster}
If $X$ is a connected and locally connected, locally compact, second-countable metrisable space, $Y$ is a locally compact, second-countable metrisable space, $f:X\to Y$, $D=D(f)$ and the graph of $f$ is closed, connected and locally connected, then for each $x\in D$ there is an open in the graph topology $U$ and a locally connected continuum $E$ such that 
\begin{enumerate}
\item $(x,f(x))\in U\subset E\subset f$,
\item $U\cap f|D = U\cap f|\partial \pi(E)$ (so $x\in \partial \pi(E)$),
\item if $X\setminus \pi(E)\ni x_n \to x$, then $f(x_n)\to \emptyset$,
\item $X\setminus \pi(E)$ has finitely many connected components. 
\end{enumerate}
\end{theorem}
\begin{proof}
Notice that $f$ is a connected and locally connected, locally compact, second-countable metrisable subspace of $X\times Y$.
Take an arbitrary $x\in D$. Choose open in the graph topology set $U$, such that $(x,f(x))\in U$ and $Clo_f(U)$ is a continuum.
By Theorem \ref{finitec}, there exists a locally connected continuum $E\subset f$ such that $f\setminus E$ has finitely many connected components and $U \subset E$.
By Lemma \ref{algebra}, $f|D \cap U = f|\partial p(E) \cap U$ and for any subsequence $X\setminus \pi(E)\ni x_n \to x$ we have $f(x_n)\to \emptyset$. By Fact \ref{lcimg}, $p(E)$ is a locally connected continuum and since $f\setminus E$ has finitely many connected components and $\pi$ is continuous, the set $\pi(f\setminus E) = \pi(f)\setminus \pi(E) = X\setminus \pi(E)$ also has finitely many connected components.
\end{proof}

Theorem 13 has an interesting consequence for $X=\R^2$, namely $\partial \pi(E)$ from the above theorem is locally connected, which implies (by Theorem \ref{lcimg}) that $f|D$ has a locally connected graph. To prove this, let me refer to the following theorem:

\begin{theorem}\label{lcborder}
If $A$ is a locally connected continuum in $\R^2$, and $S$ is a connected component of $\R^2 \setminus A$, then $\partial S$ is a locally connected continuum.
\end{theorem}

\begin{proof}
Since $\R^2$ is homeomorphic with a unit sphere without one point, it's enough to apply \cite[61.II.4]{Ku2}. 
\end{proof}

Let's formulate a simple consequence of the above.

\begin{theorem}
If $A$ is a locally connected continuum in $\R^2$ and $\R^2\setminus A$ has finitely many connected components, then $\partial A$ is locally connected.
\end{theorem}

\begin{proof}\label{lcborder_full}
Let $S_1, S_2, \dots S_n$ be connected components of $\R^2 \setminus A$.
$S_1, S_2, \dots S_n$ are open, since $\R^2\setminus E$ is an open subset of a locally connected space and thus locally connected.
Since we're dealing only with a finite number of open sets, the below equation holds.
$$
\partial A = \partial (\R^2 \setminus A) = \partial (\bigcup_{i=1}^n S_i) = \bigcup_{i=1}^n \partial S_i.
$$
By Theorem \ref{lcborder}, $\partial S_i$ is locally connected for $i=1,2, \dots, n$.
Therefore, by Fact \ref{clo_lc}, $\partial A$ is locally connected.
\end{proof}

By applying Theorem \ref{cluster}, Theorem \ref{lcborder_full} and  Theorem \ref{lcimg}, we immediately get the following corollary.

\begin{corollary}\label{main}
If $Y$ is a locally compact, second-countable metrisable space, $f:\R^2\to Y$ has a closed, connected and locally connected graph, then $f|D_f$ has a locally connected graph.
\end{corollary}

One might propose that as the local connectedness of $f|D_f$ is a local property, it might be enough to assume only local connectedness of $f$. Unfortunately, there is a simple example that shows that the connectedness of $f$ is necessary in Corollary \ref{main} and Theorem \ref{cluster}.

\begin{example}
Let $r_n = \frac{1}{4n(n+1)}$, $B_n = \{(x,y)\in \R^2: \sqrt{x^2 + (y-\frac{1}{n})^2} < r_n\}$.
$f(x,y) = 
\begin{cases}
0 \text{ for } y \geq 0 \text{ and } (x,y)\not\in \bigcup_{n=1}^\infty B_n,\\
\frac{1}{y} \text{ for } y < 0,\\
n + \tan(\frac{\pi}{2r_n}\sqrt{x^2 + (y-\frac{1}{n})^2}\;) \text{ for } (x,y)\in B_n \text{ for } n=1, \dots 
\end{cases}$ 
\end{example}

Note that in the above example $B_n$ is a sequence of pairwise disjoint open discs convergent to the point $(0,0)$. $f=0$ on the whole half plane $\R \times [0, \infty)$ except discs $B_n$. $f\geq n$ on $B_n$ and converges to infinity on $\partial B_n$. Therefore, it is easy to notice that the graph of $f$ is closed and not connected. The local connectedness of the graph is obvious everywhere except the point $(0,0,0)$. But as $f\geq n$ on $B_n$ and $f=0$ on $\R \times [0, \infty) \setminus \bigcup_{n=1}^\infty B_n$, it's enough to see that $\R \times [0, \infty) \setminus \bigcup_{n=1}^\infty B_n$ is locally connected at the point $(0,0)$. Thus the graph of $f$ is locally connected. However, $f|D_f = (\R\times \{0\} \cup \bigcup_{n=1}^\infty \partial B_n)\times \{0\}$ and is not locally connected in $(0,0,0)$. It's also easy to notice that for any open in the graph topology set $U$ such that $(0,0,0)\in U$ and for any locally connected continuum $E$ such that $U\subset E\subset f$, $\R^2 \setminus \pi(E)$ has infinitely many connected components, since $B_n\cap \pi(E) = \emptyset$ and $\partial B_n\subset E$ for almost all $n$. 

\end{document}